\documentclass[11pt]{amsart}

\usepackage{amsmath,amscd,amssymb,latexsym,color}
\usepackage{tikz-cd}

\newcommand{\sm}{\left(\smallmatrix}
\newcommand{\esm}{\endsmallmatrix\right)}
\newcommand{\mat}{\begin{pmatrix}}
\newcommand{\emat}{\end{pmatrix}}
\renewcommand{\c}{\mathfrak{c}}
\newcommand{\f}{\mathfrak{f}}

\renewcommand{\t}{\tau}

\renewcommand{\a}{\alpha}

\renewcommand{\i}{\infty}
\newcommand{\G}{\Gamma}
\newcommand{\g}{\gamma}

\newcommand{\lt}{\left}
\newcommand{\rt}{\right}
\newcommand{\Q}{\mathbb Q}
\newcommand{\Z}{\mathbb Z}
\newcommand{\C}{\mathbb C}
\newcommand{\R}{\mathbb R}

\renewcommand{\H}{\mathbb H}
\renewcommand{\k}{\kappa}

\newcommand{\F}{\mathcal F}

\newtheorem{thm}{Theorem}

\newtheorem{cor}[thm]{Corollary}

\theoremstyle{definition}

\numberwithin{equation}{section}
\numberwithin{thm}{section}

\begin{document}

\title[Hecke Equivariance]{Hecke equivariance of generalized Borcherds products of type $O(2,1)$}

\author{Daeyeol Jeon, Soon-Yi Kang and Chang Heon Kim}
\address{Department of Mathematics Education, Kongju National University, Gongju, 32588 South Korea}
\email{dyjeon@kongju.ac.kr}
\address{Department of Mathematics, Kangwon National University, Chuncheon, 24341 South Korea} \email{sy2kang@kangwon.ac.kr}
\address{Department of Mathematics, Sungkyunkwan University,
 Suwon, 16419 South Korea}
\email{chhkim@skku.edu}

\begin{abstract} Recently, a weak converse theorem for Borcherds' lifting operator of type $O(2,1)$ for $\G_0(N)$ is proved and the logarithmic derivative of a modular form for $\G_0(N)$  is explicitly described in terms of the values of Niebur-Poincaré series at its divisors in the complex upper half-plane.  In this paper, we prove that the generalized Borcherds' lifting operator of type $O(2,1)$  is Hecke equivariant under the extension of Guerzhoy's multiplicative Hecke operator on the integral weight meromorphic modular forms and the Hecke operator on  half-integral weight vector-valued harmonic weak Maass forms. Additionally, we show that the logarithmic differential operator is also Hecke equivariant under the multiplicative Hecke operator and the Hecke operator on integral weight meromorphic modular forms.  As applications of Hecke equivariance of the two operators, we obtain relations for twisted traces of singular moduli modulo prime powers and congruences for twisted class numbers modulo primes, including those associated to genus $1$ modular curves. 
\end{abstract}
\maketitle

\renewcommand{\thefootnote}%
             {}
 {

\section{Introduction}
Borcherds isomorphism, established in his seminar paper \cite[Theorem 14.1]{Bor95}, is a correspondence between the additive group of weakly holomorphic modular forms of weight $1/2$ with respect to $\G_0(4)$ which satisfy Kohnen's plus-space condition and the multiplicative group of integral weight meromorphic modular forms for some character of $\mathrm{SL}_2(\Z)$ with integer coefficients and leading coefficient $1$ with Heegner divisors. Borcherds raised several questions in connection with this theorem in \cite[17.10 and 17.11]{Bor95}, including how to generalize the isomorphism to the level greater than one and whether his isomorphism commutes with the action of Hecke operators.

The former was partially answered by himself in \cite{Bor98}, in which he simplified the proofs of his earlier results and extended them for both weight and level. To be more precise, he constructed a multiplicative lifting map from weakly holomorphic modular forms of weight $1-n/2$ for the Weil representation associated to an even lattice $L$ of signature $(2,n)$ to meromorphic modular forms for the orthogonal group of $L$ whose divisors are linear combinations of Heegner divisors. These meromorphic modular forms have infinite product expansions so called Borcherds products  if the underlying quadratic space contains $\Q$-isotropic elements. Borcherds then asked in \cite[Problem 16.10]{Bor98} whether one can reverse this lifting and it has been answered affirmatively for a large class of arithmetic subgroups $\G$ of the orthogonal group of $L$. In \cite[Theorem 5.12]{Bruinier02} and \cite[Theorem 1.2]{Bruinier14}, Bruinier proved that if $n\geq 2$, the converse theorem holds true if $L$ splits off the orthogonal sum of two hyperbolic planes over $\Z$. Even so, the Borcherds lifting is not surjective for $n=1$, i.e., when $\G=\G_0(N)$, but in \cite{BS}, Bruinier and Schwagenscheidt recently established a weak converse theorem for Borcherds products on $\G_0(N)$. They proved that if $F$ is a meromorphic modular form for $\Gamma_0(N)$ with a unitary multiplier system whose divisor on the complex upper half-plane $\H$ is a linear combination of Heegner divisors and whose divisor at the cusps is defined over the rational number field and invariant under the Fricke involution, then $F$ is (up to a non-zero constant factor) the generalized Borcherds product associated to a unique harmonic Maass form of weight $1/2$.

The latter question was answered by Guerzhoy \cite{Guer} who succinctly proved that the Borcherds isomorphism is Hecke equivariant if one considers a multiplicative Hecke operator acting on the integral weight meromorphic modular forms and the usual Hecke operator acting on the half integral weight forms. In the present paper, we generalize this result to higher level by exploiting the weak converse theorem for Borcherds products on $\G_0(N)$ introduced earlier. 

On the other hand, the well-known denominator formula for the Monster Lie algebra is derived from the logarithmic derivative of the classical $j$-invariant.  In \cite{BKO}, Bruinier, Kohnen and Ono  showed that the logarithmic derivative of a meromorphic modular form $f$ on $\text{SL}_2(\Z)$ is expressible in terms of the values of Faber polynomials of $j(\t)$ at points $\t$ in the divisor of $f$, which generalizes the denominator formula.  These results have been generalized by Bringmann et al. \cite{BKLOR} and Choi,
Lee and Lim \cite{CLL} to Niebur-Poincaré harmonic weak Maass functions of arbitrary level $N$. They proved that the logarithmic derivative of a meromorphic modular form for $\G_0(N)$ is explicitly described in terms of the values of Niebur-Poincaré series at its divisors in $\H$. In this paper, we verify that this logarithmic differential operator defined on the multiplicative group of meromorphic modular forms for $\G_0(N)$ with a unitary multiplier system (possibly of infinite order) is also Hecke equivariant under the extension of Guerzhoy's multiplicative Hecke operator and the Hecke operator on the weight $2$ meromorphic modular forms for $\G_0(N)$.

  As applications, we first obtain relations for traces of singular moduli modulo prime powers similar to those found in \cite{Ahlgren}. In another application, we establish a similar relation modulo prime for the twisted traces defined in terms of the canonical basis of the space of the weakly holomorphic modular functions associated to modular curves of genus 1 of prime levels. The second application offers a congruence relation for the corresponding twisted class numbers modulo primes.

In the following section, we will introduce vector-valued harmonic Maass forms, generalized Borcherds products, and the logarithmic differential operator, along with results that will be used later. Then, in Section 3, we demonstrate the Hecke equivariance of the generalized Borcherds lifting and the logarithmic differential operator, and in Section 4, we discuss its applications.

\section{Preliminaries}\label{sec2}

\subsection{Generalized Borcherds Product}
We recall from \cite{Bor98, BF, BO, BS} some fundamental facts regarding vector valued harmonic Maass forms for the Weil representation. Let $N$ be a positive integer and $V$ be the rational quadratic space
$$V:=\{X\in \text{Mat}_2(\Q)| \text{tr}X=0\}$$ 
with the quadratic form $Q(X)=-N{\textrm{det}}(X)$. The corresponding bilinear form is given by $(X,Y)=N\text{tr}(XY)$ for $X,Y\in V$ and the signature of $V$ is $(2,1)$.
We choose the lattice 
$$L_N:=\left\{\mat b&-a/N\\c&-b \emat: a,b,c\in \Z\right\},$$ which has  level $4N$ and the dual lattice  
$$L'_N:=\left\{\mat b/2N&-a/N\\c&-b/2N \emat: a,b,c\in \Z\right\}.$$
The discriminant group $L'_N/L_N$ is identified with $\Z/2N\Z$ with quadratic form $Q(\gamma)=\g^2/4N$. We denote the standard basis elements of the group algebra $\C[L'_N/L_N]$ by $\mathfrak e_\g$ for $\g\in L'_N/L_N$, and write $\langle\cdot,\cdot\rangle$ for the standard scalar product, anti-linear on the second entry, such that $\langle \mathfrak e_\g,\mathfrak e_{\g'}\rangle=\delta_{\g,\g'}$. The weil representation $\rho_N$ associated with $L'_N/L_N$ is a unitary representation of the integral metaplectic group $\text{Mp}_2(\Z)$ on $\C[L'_N/L_N]$. We denote by $\bar{\rho}_N$ the corresponding dual Weil representation. 

A smooth function $f:\H\to\C[L'_N/L_N]$ is called a harmonic weak Maass form of weight $k$ for $\rho_N$ for the group $\text{Mp}_2(\Z)$ if 
\begin{enumerate}
\item $f$ transforms like a modular form of weight $k$ for $\rho_N$ under $\text{Mp}_2(\Z)$, 
\item $\Delta_k(f)=0$, where $\Delta_k$ is the weight $k$ hyperbolic Laplace operator,
\item $f$ is at most linear exponential growth at the cusp $[i\i]$.
\end{enumerate}
Let $H_{k,\rho_N}$ be the space of $\C[L'_N/L_N]$-valued harmonic weak Maass forms of weight $k$ and type $\rho_N$.
The Fourier expansion of any $f\in H_{k,\rho_N}$ gives a unique decomposition $f=f^++f^-$, where
$$f^+(\t)=\sum_{\g\in L'_N/L_N}\sum_{\substack{n\in\Q\\n\gg-\i}}c_f^+(n,\g)q^{\frac{n}{4N}}\mathfrak e_\g,$$
$$f^-(\t)=\sum_{\g\in L'_N/L_N}\sum_{\substack{n\in\Q\\ n<0}}c_f^-(n,\g)\G\lt(1-k,\frac{\pi |n| \text{Im}(\t)}{N}\rt)q^{\frac{n}{4N}}\mathfrak e_\g,$$
where $q:=e^{2\pi i \t}$ for $\t\in\H$ and  $\G(a,x)=\int_x^\i e^{-t}t^{a-1}dt$ is the incomplete gamma function. One sees from the transformation property for $\rho_N$ that $c_f^\pm(n,-\g)=(-1)^{k-1/2}c_f^\pm(n,\g)$ for all $n\in\Z$, $\g\in L'_N/L_N$, and $c_f^\pm(n,\g)=0$ unless $n\equiv \g^2\pmod{4N}$.

Let $\Delta$ be a fundamental discriminant and let $r\in\Z$ such that $\Delta\equiv r^2\pmod{4N}$. Also put
$$\tilde{\rho}_N=\begin{cases}\rho_N, \ \text{if}\ \Delta>0,\\
\bar{\rho}_N,\ \text{if}\  \Delta<0.\end{cases}$$
Bruinier and Ono established the following generalized Borcherds lift for harmonic weak Maass forms.
\begin{thm}\cite[Theorem 6.1]{BO}\label{bothm}
Let $f\in H_{k,\tilde{\rho}_N}$ be a harmonic weak Maass form with real coefficients $c^+(m,\g)$ for all $m\in\Q$ and $\g\in L'_N/L_N$. Moreover $c^+(n,\g)\in \Z$ for all $n\leq 0$. Then the infinite product
\begin{equation}\label{gbp}
\Psi_{\Delta,r}(\t,f)=q^{\nu_f}\prod_{n>0}\prod_{b(\Delta)}\lt(1-\zeta_{\Delta}^b q^n\rt)^{(\frac{\Delta}{b})c^+_f(\Delta n^2,rn)},
\end{equation}
where $\zeta_\ell:=e^{2\pi i/\ell}$,  $\nu_f\in\R$ is the Weyl vector at the cusp $[i\i]$ associated with $f$, converges for $\text{Im}(\t)$ sufficiently large and has a meromorphic continuation to all of $\H$ with the following properties:
\begin{enumerate}
\item It is a meromorphic modular form for $\G_0(N)$ with a unitary character $\sigma$ which may have infinite order.
\item The weight of $\Psi_{\Delta,r}(\t,f)$ is $c_f^+(0,0)$ when $\Delta=1$, and is $0$ when $\Delta\neq 1$.
\item The divisor of $\Psi_{\Delta,r}(\t,f)$ on $Y_0(N):=\G_0(N)\backslash\H$ is given by the linear combination
$$Z_{\Delta,r}(f)=\sum_{\g\in L'_N/L_N}\sum_{n<0}c_f^+(n,\g)Z_{\Delta,r}(n,\g)$$ of the twisted Heegner divisors $Z_{\Delta,r}(n,\g)$ that are defined in \cite[Section 5]{BO}.
\end{enumerate}
\end{thm}
\noindent The Jacobi symbol $(\frac{\Delta}{b})$ in \eqref{gbp} is defined by $1$ if $\Delta=1$ and $0$ otherwise, when $b=0$. We note that $\Psi:=\Psi_{\Delta,r}(\t,f)$ is a map from $H_{1/2,\tilde{\rho}_N}$ to  the multiplicative group of integer weight meromorphic modular forms for $\G_0(N)$ with a unitary multiplier system (possibly of infinite order).

Recently, Bruinier and Schwagenscheidt proved that this lifting is reversible.
Let $\mathcal M(N)$ be the multiplicative group of meromorphic modular forms $F$ for $\G_0(N)$ with a unitary multiplier system (possibly of infinite order). When $\Delta=1$, we suppose that the divisor of $F$ on $Y_0(N)$ is a linear combination of Heegner divisors and whose divisor at the cusps is defined over the rational number field and invariant under the Fricke involution. If $\Delta\neq 1$, we suppose that the divisor of $F$ on $X_0(N):=\G_0(N)\backslash(\H\cup\Q\cup\{[i\i]\})$ is a linear combination of the twisted Heegner divisors $Z_{\Delta,r}(n,\g)$. Let $\mathcal{M}'(N)$ be the subgroup of $\mathcal M(N)$ composed of meromorphic modular forms satisfying these conditions. 
\begin{thm}\cite[Theorem 1.1, Theorem 1.5]{BS}\label{conv}
 If $F\in\mathcal M'(N)$, then there exists a harmonic Maass form $f\in H_{1/2,\tilde{\rho}_N}$ whose (twisted) generalized Borcherds product $\Psi_{\Delta, r}(f,z)$ is a non-zero constant multiple of $F$. 
\end{thm}

\subsection{Logarithmic Derivative}
From now on, we assume $k\in \Z$ and $\k\in\frac12+\Z$. Ramanujan's differential operator
$\Theta:=q\frac{d}{dq}=\frac{1}{2\pi i}\frac{d}{d\t}$ satisfies
$$\Theta\left(\sum_{n=h}^\i a(n)q^n\right)=\sum_{n=h}^\i na(n)q^n.$$
For a meromorphic modular form $f$ of weight $k$ on $\G_0(N)$, the Serre derivative $\displaystyle{\partial_k(f):=\Theta(f)-\frac{kE_2f}{12}}$ preserves modularity with weight raised by $2$, where $E_k(\t)$ is the Eisenstein series of weight $k$. 
We let $M^{mero}_k(N)$ ($M_k(N)$, resp.) denote the space of meromorphic  (holomorphic, resp.) modular forms of weight $k$ on $\G_0(N)$ and define a lifting (in fact, the logarithmic derivative) $\mathfrak {D}: \mathcal M(N)\to M^{mero}_2(N)$ by 
\begin{equation*}\label{defDV}
\displaystyle{\mathfrak{D}(f):=\frac{\partial_k(f)}{f}}=\frac{\Theta(f)}{f}-\frac{kE_2}{12}.
\end{equation*}

As an example, consider the elliptic modular $j$-function 
$$j(\t):=q^{-1}+744+196884q+21493760q^2+\cdots.$$
For the normalized Hecke operator $T(n)$,
the weakly holomorphic modular functions given by
\begin{equation*}\label{hj}J_{n}(\t):=(j(\t)-744)|T(n), \quad (n\geq 1)
\end{equation*}
and $J_0(\t):=1$ form a basis of the space of weakly holomorphic modular functions on $\G_0(1)$. 
For every $\t\in\H$, Asai, Kaneko and Ninomiya  \cite[Theorem 3]{AKN} proved that
\begin{equation}\label{df}
\mathfrak{D}(j(\t)-j(z))=\frac{1}{2\pi i}\frac{j'(\t)}{j(\t)-j(z)}=-\sum_{n=0}^\i J_{n}(z) q^n.
\end{equation}
In \cite[Theorem 1]{BKO}, Bruinier, Kohnen and Ono generalized this by proving that if 
\begin{equation*}\label{f}
f(\t)=q^h+\sum_{n=h+1}^\i a_f(n)q^n
\end{equation*} is a non-zero weight $k$ meromorphic modular form on $\G_0(1)$, then 
\begin{equation*}\label{serco}
\mathfrak{D}(f)=h-\frac{k}{12}-\sum_{n=1}^\i \lt(\sum_{z\in \F_1}\frac{\hbox{ord}_{z}(f)}{e_{1,z}} J_{n}(z)\rt)q^n\in M_2^{mero}(1),
\end{equation*}
%{\color{green}and
%\begin{equation}\label{18} \sum_{d|n}c(d)d-2k\sigma_1(n)=\sum_{z\in \F_1}e_{1,z}\hbox{ord}_{z}(f) J_{n}(z).
%\end{equation}
%$\sigma_1(n)$ is the sum of divisors of $n$,} 
where $\F_N$ is the fundamental domain of $\G_0(N)$ and $e_{N,z}$ is the ramification index of the covering $\H\rightarrow \G_0(N)\backslash\H$ at $z$.
%cardinality of $G_0(N)_{z}/\{\pm1\}$ for each $z \in \H$, where $\G_0(N)_{z}$ denotes the stabilizer of $z$ in $\G_0(N)$. 

Bringmann et al. \cite{BKLOR} and Choi, Lee, and Lim \cite{CLL} have extended these results to Niebur-Poincaré harmonic weak Maass functions $J_{N,n}$ of arbitrary level $N$ by proving that the logarithmic derivative of a meromorphic modular form for $\G_0(N)$ is described explicitly in terms of the values of  $J_{N,n}$ at its divisors in $\H$.
In \cite[Theorem 2.1]{JKK-divisor}, the authors showed that  the logarithmic derivative of a meromorphic modular form for $\G_0(N)$ can be expressed explicitly in terms of the values of weakly holomorphic modular functions.
Let $M_k^!(N)$ denote the space of scalar-valued weakly holomorphic modular forms of weight $k$ with respect to $\G_0(N)$ and $M_k^\sharp(N)$ be the subspace of $M_k^!(N)$ with poles allowed only at the cusp $[i\infty]$. We assume $[i\infty]$ is not the Weierstrass point.
When $g_N$ is the genus of $X_0(N)$, the canonical basis elements of $M_0^\sharp(N)$,  denoted by $\mathfrak{f}_{N,m}$, have the  Fourier expansions of the form
\begin{equation*}\label{fn}
\mathfrak{f}_{N,m}(\t)=q^{-m}+\sum_{\ell=1}^{g_N} a_{N}(m,-\ell)q^{-\ell}+\sum_{n=1}^{\i} a_{N}(m,n)q^{n},\ (m\geq g_N+1).
\end{equation*}
In addition, $\f_{N,0}=1$ and $\f_{N,m}=0$ for $1\leq m\leq g_N$. (See \cite{JKK-hecke} for the construction of the basis in detail.)
 For every point $z\in\H$, we define
\begin{equation*}
{F}_{N,z}(\t):=\sum_{n=0}^\i \f_{N,n}(z)q^n=1+\sum_{n=g+1}^\i \f_{N,n}(z)q^n.
\end{equation*}
According to \cite[Theorem 2.1]{JKK-divisor}, for a meromorphic modular form $f$ of weight $0$ on $\G_0(N)$, one has
\begin{equation*} \label{divisor}
\sum_{z\in \F_N} \frac{{\rm ord}_z(f)}{e_{N,z}} F_{N,z}(\t)+\frac{\Theta (f)}{f}\in M_2(N).
\end{equation*}

\section{Hecke equivariance}

We first define the action of Hecke operators on half-integral weight vector-valued harmonic weak Maass forms by their action on Fourier expansions following \cite[Section 7]{BO} and \cite[Section 2]{BS}. If $g\in H_{\k,\tilde{\rho}_N}$ and its holomorphic part has the Fourier expansion $g^+=\sum_{\g,n}b^+_g(n,\g)q^{\frac{n}{4N}}\mathfrak e_\g$, then for a prime $p$ not dividing $N$, the Fourier expansion of the holomorphic part of $g|T_\k(p^2)$ is given by
$$(g|T_\k(p^2))^+=\sum_{\g,n} b^+_{g|T_\k}(n,\g)q^{\frac{n}{4N}}\mathfrak e_\g,$$ 
where
\begin{equation}\label{genhecke}
b^+_{g|T_\k}(n,\g)=b^+_g(p^2n,p\g)+p^{\k-3/2}\left(\frac{\sigma n}{p}\right)b^+_g(n,\g)+p^{2\k-2}b^+_g(n/p^2,\g/p).
\end{equation}
Here $\sigma=1$ if $\tilde{\rho}_N=\rho_N$, and $\sigma=-1$ if $\tilde{\rho}_N=\bar\rho_N$.

Next we define a multiplicative Hecke operator acting on $\mathcal M(N)$, the multiplicative group of integer weight meromorphic modular forms for $\G_0(N)$ with a unitary multiplier system (possibly of infinite order).
We let $\mathcal M_{k}(N)\subset\mathcal M(N)$ denote the subset that consists of modular forms of weight $k$. %and for which {\color{red}the order at the cusp  $[i\i]$ is $-h$.}
  For a matrix $\a=\sm a&b\\c&d\esm$ with a positive determinant, the action of weight $k$ slash operator on meromorphic functions on $\H$ is given by
$$f|_k[\a]:=f\lt(\frac{a\t+b}{c\t+d}\rt)(c\t+d)^{-k}(\text{det}\a)^{k/2}.$$
Let $p$ be a prime and $f\in \mathcal M_{k}(N)$.  Following Guerzhoy \cite{Guer}, we define the multiplicative Hecke operator $\mathcal T(p)$ acting on $\mathcal M(N)$ by
$$f|\mathcal T(p):=f|\mathcal T_{k}(p):=\varepsilon p^{k(p-1)/2}\prod_{i=1}^{p+1}f|_k[\alpha_i],$$
where $\varepsilon$ is a constant chosen so that the leading coefficient of $f|\mathcal T(p)$ is $1$ and  $\alpha_i$'s are the right coset representatives of the action of $\G_0(N)$ for the set of $2\times 2$ matrices with integer entries and determinant $p$. 
Then $\mathcal T(p)$ maps elements of $\mathcal M_{k}(N)$ to elements of $\mathcal M_{(p+1)k}(N)$. 

Let $M^!_{k,\rho_N}$ be the subspace of vector-valued weakly holomorphic modular forms of  $H_{k,{\rho}_N}$ and $M^!_{1/2}(4N)$ denote the space of scalar-valued weakly holomorphic modular forms of weight $1/2$ with respect to $\G_0(4)$ which satisfy Kohnen's plus-space condition.
Borcherds established in \cite[Theorem 14.1]{Bor95} an isomorphism between the additive group $M^!_{1/2}(4)$  and the multiplicative group $\mathcal M(1)$. 
Recall the generalized Borcherds lifting ${\Psi}:H_{1/2,\tilde{\rho}_N} \to \mathcal M(N) $ in Theorem \ref{bothm}. We generalize the Hecke equivariance of Borcherds isomorpshim in \cite{Guer} to that of generalized Borcherds products.

\begin{thm}\label{HB}
Let $H'_{\k,\widetilde{\rho}_N}$ be the additive subgroup of $H_{\k,\widetilde{\rho}_N}$ consisting of forms that satisfy the conditions in Theorem \ref{bothm}. Then the following diagrams of groups and their homomorphisms are commutative when $p$ is a prime not dividing $N$ nor the discriminant $\Delta$:
$$\begin{tikzcd}
H'_{1/2,\tilde{\rho}_N} \arrow[r,"\Psi"] \arrow[d, "pT_{1/2}(p^2)"] & \mathcal M(N)  \arrow[d, "\mathcal T(p)"] \\
H'_{1/2,\tilde{\rho}_N} \arrow[r,"\Psi"] & \mathcal M(N) \\
\end{tikzcd}$$
\end{thm}

\begin{proof} 
Let $\phi\in H'_{1/2,\tilde{\rho}_N}$. Then 
 \begin{equation*}\label{Bp}
f:=\Psi(\phi)=\Psi_{\Delta,r}(\t,\phi)=q^{\nu_\phi}\prod_{n>0}\prod_{b(\Delta)}\lt(1-\zeta_\Delta^bq^n\rt)^{(\frac{\Delta}{b})c^+_\phi(|\Delta| n^2,rn)}\in\mathcal M(N).
\end{equation*}
We let
\begin{align*}
c_b(n,r)&=\left(\frac{\Delta}{b}\right)c_\phi^+(|\Delta|n^2,rn),\\ 
c^*_b(n,r)&=\left(\frac{\Delta}{b}\right)c_{\phi|{T_k}}^{+}(|\Delta|n^2,rn).
\end{align*}
By the definition of $\mathcal T(p)$, we have
\begin{align}\label{mpt}
f|\mathcal T(p)=&\varepsilon f|_k \left[\sm p&0\\0&1\esm\right]\prod_{j=0}^{p-1}f|_k\left[\sm 1&j\\0&p\esm\right]\\
=&\varepsilon q^{\nu_\phi p}\prod_{j=0}^{p-1}\zeta_p^{\nu_\phi j}q^{\nu_\phi/p}\prod_{n\geq 1}\prod_{b(\Delta)}(1-\zeta_\Delta^bq^{pn})^{c_b(n,r)}
\prod_{n\geq 1}\prod_{b(\Delta)}\prod_{j=0}^{p-1}(1-\zeta_p^{nj} \zeta_\Delta^bq^{n/p})^{c_b(n,r)}\cr
=&q^{\nu_\phi (p+1)}\prod_{n\geq 1}\prod_{b(\Delta)}(1-\zeta_\Delta^bq^{pn})^{c_b(n,r)}(1- \zeta_\Delta^bq^{n})^{pc_b(pn,r)}
\prod_{{n\geq 1}\atop{p\nmid n}}\prod_{b(\Delta)}(1-\zeta_\Delta^{pb}q^{n})^{c_b(n,r)}\cr
=&q^{\nu_\phi (p+1)}\prod_{n\geq 1}\prod_{b(\Delta)}(1-\zeta_\Delta^bq^{n})^{pc^*_b(n,r)}=\Psi(\phi|pT_\frac12(p^2)),\nonumber
\end{align}
where the last equality follows from \eqref{gbp} and the penultimate equality follows from \eqref{genhecke} along with the fact
\begin{equation*}\label{bp}
\prod_{{n\geq 1}\atop{p\nmid n}}\prod_{b(\Delta)}(1-\zeta_\Delta^bq^{n})^{\left(\frac{\Delta n^2}{p}\right)c_b(n,r)}
=\prod_{{n\geq 1}\atop{p\nmid n}}\prod_{b(\Delta)}(1-\zeta_\Delta^{pb}q^{n})^{\left(\frac{\Delta}{b}\right)c_{\phi}^{+}(|\Delta|n^2,rn)}.
\end{equation*}

%Utilizing \eqref{genhecke} and \eqref{bp}, we find that the right hand side of \eqref{mpt} is equal to
%\begin{align}
%=&q^{\rho_\phi (p+1)}\prod_{{n\geq 1}\atop{p\nmid n}}\prod_{b(\Delta)}%[1-e(b/\Delta)q^{pn}]^{c_b(n,r)}[1-e(pb/\Delta)q^{n}]^{c_b(n,r)}\\
%&\prod_{n\geq 1}\prod_{b(\Delta)}[1-e(b/\Delta)q^{p^2n}]^{c_b(pn,r)}[1- %e(b/\Delta)q^{n}]^{pc_b(pn,r)}\cr
%&q^{\rho_\phi (p+1)}\prod_{n\geq 1}\prod_{b(\Delta)}[1-e(b/\Delta)q^{n}]^{pc^*_b(n,r)}=B(\phi|pT_\frac12(p^2)),\nonumber
%\end{align}

This proves that
\begin{equation*}\label{gbhe}
\Psi(\phi)|\mathcal T(p)=\Psi(\phi|pT_{1/2}(p^2)).
\end{equation*}

\end{proof}

Now we show that the divisor lifting $\mathfrak D$ also commutes with the action of the Hecke operators $\mathcal T$ and $T_2$, where $T_k$ acts on
a meromorphic modular form $g=\sum_{n\gg-\i}c(n)q^n$ of integral weight $k$ by
\begin{equation}\label{ch}
g|T_k(p)=\sum_{n\gg-\i}\lt(c(pn)+p^{k-1}c(\frac{n}{p})\rt)q^n.
\end{equation}
Recall from Theorem \ref{conv} that $\mathcal{M}'(N)$ is the subgroup of $\mathcal M(N)$ consisting of meromorphic modular forms for which the generalized Borcherds product does have an inverse.
\begin{thm}\label{DC} Let $p$ be a prime not dividing $N$ nor the discriminant $\Delta$. Then the following diagrams of groups and their homomorphisms are commutative:
$$\begin{tikzcd}
\mathcal M'(N) \arrow[r,"\mathfrak{D}"] \arrow[d, "\mathcal T(p)"] & M^{mero}_2(N) \arrow[d, "T_2(p)"] \\
\mathcal M'(N) \arrow[r,"\mathfrak{D}"] & M^{mero}_2(N)  \\
\end{tikzcd}$$
\end{thm}

\begin{proof}
Let $\zeta:=\zeta_\Delta=e^{2\pi i/\Delta}$ and $f\in\mathcal M'(N)$ of weight $k$. Then by Theorem \ref{conv}, there is a harmonic weak Maass form $\phi\in H'_{1/2,\tilde{\rho}_N}$ such that
\begin{equation*}\label{Bf}
f=\Psi(\phi)=\Psi_{\Delta,r}(\t,\phi)=q^{\nu_\phi}\prod_{n>0}\prod_{b(\Delta)}\lt(1-\zeta^b q^n\rt)^{c_b(n,r)}.
\end{equation*}
Taking the logarithmic derivative of $f$ gives
\begin{align*}
\frac{\Theta(f)}{f}
%&=\rho_\phi+\sum_{b(\Delta)}\sum_{n\geq 1}c_b(n,r)q\frac{d}{dq}\lt(\log(1-\zeta^bq^n)\rt)
&=\nu_\phi-\sum_{b(\Delta)}\sum_{n\geq 1}c_b(n,r)\frac{n\zeta^bq^n}{1-\zeta^bq^n}
%&=\rho_\phi-\sum_{b(\Delta)}\sum_{m\geq 1}c_b(m,r)me(b/\Delta)q^m\left(1+e(b/\Delta)q^m+e(b/\Delta)^2q^{2m}+e(b/\Delta)^3q^{3m}+\cdots\right)\\
=\nu_\phi-\sum_{b(\Delta)}\sum_{n\geq 1}\sum_{m\geq 1}
c_b(n,r)n \zeta^{bm}q^{nm}\\
&=\nu_\phi-\sum_{b(\Delta)}\sum_{n\geq 1}\left(\sum_{m|n}c_b(m,r)m\zeta^{\frac{bn}{m}}\right)q^n
=\nu_\phi-\sum_{n\geq 1}\left(\sum_{b(\Delta)}\sum_{m|n}c_b(m,r)m\zeta^{\frac{bn}{m}}\right)q^n.\\
\end{align*}
Hence 
\begin{equation*}\label{div11}
\mathfrak D(f)=\nu_\phi-\sum_{n\geq 1}\left(\sum_{b(\Delta)}\sum_{m|n}c_b(m,r)m\zeta^{\frac{bn}{m}}\right)q^n-\frac{k}{12}E_2
\end{equation*}
and similarly from \eqref{mpt}, we have
\begin{equation*}\label{div22}
\mathfrak D(f|\mathcal T(p))=(p+1)\nu_\phi-\sum_{n\geq 1}\left(\sum_{b(\Delta)}\sum_{m|n}c^*_b(m,r)pm\zeta^{\frac{bn}{m}}\right)q^n-\frac{k(p+1)}{12}E_2.
\end{equation*}
Then by the definition of $T_2(p)$ in \eqref{ch}, we have
\begin{equation*}\label{tpdf}
\mathfrak D(f)|T_2(p)=(p+1)\nu_\phi-\frac{k(p+1)}{12}E_2+\sum_{n=1}^\i A(n)q^n,
\end{equation*}
where 
\begin{equation}\label{an}
A(n)=\sum_{b(\Delta)}\sum_{m|pn}c_b(m,r)m\zeta^{\frac{bpn}{m}}+ \sum_{b(\Delta)}\sum_{m|n/p}c_b(m,r)pm\zeta^{\frac{bn}{pm}}.
\end{equation}
So it remains to prove that 
\begin{equation}\label{anc}A(n)=\sum_{b(\Delta)}\sum_{m|n}c^*_b(m,r)pm\zeta^{\frac{bn}{m}}
\end{equation}
to finish the proof.
By \eqref{genhecke}, the right-hand side of \eqref{an} equals
\begin{equation}\label{rhs}
%\sum_{b(\Delta)}\sum_{m|n}c^*_b(m,r)pm\zeta^{\frac{bn}{m}}\\
\sum_{b(\Delta)}\sum_{m|n}\left\{c_b(pm,r)pm\zeta^{\frac{bn}{m}}
+\left(\frac{\Delta m^2}{p}\right)c_b(m,r)m\zeta^{\frac{bn}{m}} + c_b(m/p,r)m\zeta^{\frac{bn}{m}}\right\}.
\end{equation}
If $p\nmid n$,  it follows from \eqref{an} that
\begin{equation*}
A(n)=\sum_{b(\Delta)}\sum_{m|pn}c_b(m,r)m\zeta^{\frac{bpn}{m}}
=\sum_{b(\Delta)}\sum_{m|n}c_b(pm,r)pm\zeta^{\frac{bn}{m}}
+\sum_{b(\Delta)}\sum_{m|n}c_b(m,r)m\zeta^{\frac{(bp)n}{m}},
\end{equation*}
which is equal to \eqref{rhs}, because $\displaystyle{\lt(\frac{\Delta n^2}{p}\rt)\lt(\frac{\Delta}{bp}\rt)=\lt(\frac{\Delta}{b}\rt)}$, and thus \eqref{anc} holds.
Now assume $p\mid n$, and let $n=p^kn'$ with $k>0$ and $p\nmid n'$. 
Then the first sum of $A(n)$ in \eqref{an} is
\begin{align*}
&\sum_{b(\Delta)}\sum_{m|p^{k+1}n'}c_b(m,r)m\zeta^{\frac{bp^{k+1}n'}{m}}\\
&\quad =\sum_{b(\Delta)}\sum_{m|p^kn'}c_b(pm,r)pm\zeta^{\frac{bp^{k}n'}{m}}+\sum_{b(\Delta)}\sum_{m|p^kn'}\left(\frac{m^2}{p}\right)\left(\frac{\Delta^2}{p}\right)c_b(m,r)m\zeta^{\frac{(bp)p^{k}n'}{m}}\\
%=&\sum_{b(\Delta)}\sum_{m|p^kn'}c_b(pm)pme(b/\Delta)^{\frac{p^{k}n'}{m}}+\sum_{b(\Delta)}\sum_{m|n'}c_b(m)me(b/\Delta)^{\frac{p^{k}n'}{m}}\\
&\quad=\sum_{b(\Delta)}\sum_{m|n}c_b(pm,r)pm\zeta^{\frac{bn}{m}}+\sum_{b(\Delta)}\sum_{m|n}\left(\frac{\Delta m^2}{p}\right)c_b(m,r)m\zeta^{\frac{bn}{m}}
\end{align*}
and the second sum of $A(n)$ in \eqref{an} is
\begin{equation*}
\sum_{b(\Delta)}\sum_{m|p^{k-1}n'}c_b(m,r)pm\zeta^{\frac{bp^{k-1}n'}{m}}=
\sum_{b(\Delta)}\sum_{m|n}c_b(m/p,r)m\zeta^{\frac{bn}{m}}.
%=\sum_{b(\Delta)}\sum_{m|p^kn'}c_b(m/p)me(b/\Delta)^{\frac{n}{m}}.\\
\end{equation*}
%because $c\left(\frac{d^2}{p^2}\right)=0$ for $p\nmid d$.
Hence \eqref{anc} holds when $p\mid n$ as well.
Therefore,  
\begin{equation*}\label{ldhe}
\mathfrak D(f|\mathcal T(p))=\mathfrak D(f)|T_2(p).
\end{equation*}
\end{proof}

\section{Applications}
By Theorems \ref{HB} and \ref{DC}, all diagrams below are commutative:
$$\begin{tikzcd}
H'_{1/2,\tilde{\rho}_N} \arrow[r,"\Psi"] \arrow[d, "pT_{1/2}(p^2)"] & \mathcal M'(N) \arrow[r,"\mathfrak{D}"] \arrow[d, "\mathcal{T}(p)"] & M^{mero}_2(N) \arrow[d, "T_2(p)"] \\
H'_{1/2,\tilde{\rho}_N} \arrow[r,"\Psi"] & \mathcal M'(N) \arrow[r,"\mathfrak{D}"] & M^{mero}_2(N)  \\
\end{tikzcd}$$
The values by $\Psi$ may not be in $\mathcal{M}'$ but we consider only the functions that justify this diagram.
Throughout this section, $\Delta\in\Z$ is a fundamental discriminant and $r\in\Z$ such that $\Delta\equiv r^2\pmod {4N}$ and $d$ is a positive integer such that
$-d$ is congruent to a square modulo $4N$.  We denote by $\mathcal Q_{\Delta,N}$ the set of positive definite integral binary quadratic forms 
$$
Q=[a,b,Nc]=aX^2+bXY+cNY^2 \, \, (a,b,c\in \mathbb Z)
$$
of discriminant $\Delta$, with the usual action of the group $\G_0(N)$. To each $Q\in \mathcal Q_{\Delta,N}$, we associate its unique root $\alpha_Q\in \mathbb H$, called a {\it CM point}. The values assumed by the modular functions at $\a_Q$ are called {\it  singular moduli}, which play important roles in number theory. 
The genus character for 
$\lambda=\sm b/2N&-a/N\\ c &-b/2N\esm\in L'_N$ is defined by
$$\chi_\Delta(\lambda)=\chi_\Delta([a,b,Nc]):=\begin{cases}
(\frac{\Delta}{n}),\quad \text{if}\ \Delta\mid b^2-4Nac,\ \frac{b^2-4Nac}{\Delta}\equiv\square\ (4N),\ \text{and}\ \gcd(a,b,c,\Delta)=1,\\
0,\quad \text{otherwise}.
\end{cases}$$ 
Here $n$ is any integer prime to $\Delta$ represented by one of the quadratic forms $[N_1a,b,N_2c]$ with $N_1N_2=N$ and $N_1, N_2>0$. 
For $f\in M_0^!(N)$, we define the twisted trace of singular moduli
\begin{equation*}\label{tr}
Tr_{\Delta,d,N}(f):=\sum_{Q\in\mathcal Q_{d\Delta,N}/\bar{\Gamma}_0(N)}\frac{\chi_\Delta(Q)}{|\bar{\Gamma}_0(N)_Q|}f(\a_Q),
\end{equation*}
where $\bar{\G}=\G\slash\{\pm I\}$. We drop $N$ when $N=1$, i.e., $Tr_{\Delta,d}(f):=Tr_{\Delta,d,1}(f)$ and define the twisted class number by 
$$H_N(D,d):=Tr_{\Delta,d,N}(1).$$

After Zagier's influential work on traces of singular moduli \cite{Zagier}, several authors have studied congruences for $Tr_{\Delta,d}(J_n)$ modulo prime powers and congruences for class numbers. We will add new results for each of them in the subsections below.

\subsection{Twisted traces of singular moduli}

Consider the case when $N=1$ discussed in \cite[Section 8.1]{BO}. In this case, $L'_1/L_1\cong \Z/2\Z$, $H_{1/2,\bar{\rho}_1}=0$, and $H_{1/2,\rho_1}=M^!_{1/2,\rho_1}\cong M^{!}_{1/2}(4)$. The basis elements of $M^{!}_{1/2}(4)$ found by Zagier are given by
$$f_d=q^{-d}+\sum_{\substack{0<n \equiv0,1\ (4)}}A(n,d)q^n$$
for all negative discriminants $-d$. We let $\Delta>1$. Then $r\in\Z$ satisfying $\Delta\equiv r^2\pmod 4$ is either $0$ or $1$ depending on the parity of $\Delta$. By Theorem \ref{bothm}, the generalized Borcherds product of $f_d$,
$$\Psi(f_d):=\Psi_{\Delta,r}(\t,f_d)=\prod_{n=1}^\i\prod_{b=0}^{|\Delta|-1}\lt(1-\zeta_\Delta^bq^n\rt)^{(\frac{\Delta}{b})A(\Delta n^2,d)}$$
is a weight $0$ meromorphic modular form for the group $\G:=\G_0(1)$ with a unitary character. 
By identifying a vector $\lambda\in L_{\Delta d}$ of negative norm with a quadratic form $Q=[a,b,c]\in\mathcal Q_{\Delta d}$, it was shown in (\cite[Eq. (8.2)]{BO}) that 
$$\Psi(f_d)%=\prod_{\lambda\in L_{\Delta d}/\G}(j(\t)-j(Z(\lambda))^{\chi_\Delta(\lambda)}
=\prod_{Q\in \mathcal Q_{\Delta d}/\bar{\G}}(j(\t)-j(\alpha_Q))^{\chi_\Delta(Q)/|\bar{\G}_Q|}.$$
%
%Then we can write the divisor of $\Psi_\Delta(z,f_d)$ on $X_0(1)$ by
%$$Z_\Delta(d)=\sum_{Q\in\mathcal Q_{d\Delta}/\G_0(1)}%\frac{\chi_\Delta(Q)}{\omega_Q}\a_Q.$$
Hence we deduce that
\begin{equation*}\label{tr}(\mathfrak D\circ \Psi)(f_d)=\frac{1}{2\pi i}\sum_{Q\in\mathcal Q_{d\Delta,1}/\bar{\G}}\frac{\chi_\Delta(Q)}{|\bar{\G}_Q|}\frac{j'(\t)}{j(\t)-j(\a_Q)}=-\sum_{n=0}^\i \lt(\sum_{Q\in\mathcal Q_{d\Delta,1}/\bar{\G}}\frac{\chi_\Delta(Q)}{|\bar{\G}_Q|}J_n(\a_Q)\rt)q^n,
\end{equation*}
where the second equality follows from \eqref{df}. 
%Using the twisted trace of singular moduli defined by
%\begin{equation}\label{ttr}Tr_{\Delta,d}(J_n):=\sum_{Q\in\mathcal Q_{d\Delta,1}/\G}%\frac{\chi_\Delta(Q)}{|\G_Q|}J_n(\a_Q),
%\end{equation}
Thus we may write  the composition $\mathfrak D\circ \Psi$ lift of $f_d$ as the generating function of the twisted trace of singular moduli:
\begin{equation*}\label{gtr}(\mathfrak D\circ \Psi)(f_d)=-\sum_{n=0}^\i Tr_{\Delta,d}(J_n)q^n.
\end{equation*}
Let $p$ be a prime and $m$ be a non-negative integer.  For a modular form $g=\displaystyle\sum_{n\geq n_0}^{\i}c(n)q^n$ of weight $k$, it is found in \cite{Sh} that
\begin{equation}\label{Sh}
g|T_k(p^{m})=p.p. + \sum_{n\geq 1}^{\i}\sum_{t=0}^{\ell} p^{(k-1)t}c\left(\frac{p^m n}{p^{2t}}\right),
\end{equation}
where $p.p$ represents the principal part and $\ell={\rm min}({\rm ord}_p(n),m)$. By \eqref{gtr} then, we deduce
\begin{equation}\label{ap}(\mathfrak D\circ \Psi)(f_d)|T_2(p^m)=-\sum_{n=0}^\i \left(\sum_{t=0}^{\ell} p^{t}Tr_{\Delta,d}\left(J_\frac{p^m n}{p^{2t}}\right)\right)q^n. 
\end{equation}

On the other hand, if we write $d=p^{2u}d'$, where $p^2\nmid d'$ and $u\geq 0$, then by \cite[Proposition 2.4]{BeF16}, we have 
\begin{equation*}\label{BeF}
f_d|p^mT_{1/2}(p^{2m})=\begin{cases} \displaystyle\sum_{t=0}^m p^{m-t}f_{p^{2u-2m+4t}d'},\hbox{ if } 0\leq m<u,\\
\displaystyle\sum_{t=0}^{m-u} \left(\frac{-d'}{p}\right)^{m-u-t}p^uf_{p^{2t}d'}+\displaystyle\sum_{t=1}^{u}p^{u-t}f_{p^{2m-2u+4t}d'},\hbox{ if } m\geq u.
\end{cases}
\end{equation*}
Hence we obtain from \eqref{gtr} that 
\begin{equation}\label{ft}
(\mathfrak D\circ \Psi)(f_d|p^mT_{1/2}(p^{2m}))=-\sum_{n=0}^\i T(\Delta, d,n) q^n,
\end{equation}
where
\begin{equation*}\label{ftcof}
T(\Delta, d,n)=
\begin{cases} 
\displaystyle{ \sum_{t=0}^m p^{m-t}  Tr_{\Delta,p^{2u-2m+4t}d'}(J_n)},\hbox{ if } 0\leq m<u,\\
\displaystyle{\sum_{t=0}^{m-u}  \left(\frac{-d'}{p}\right)^{m-u-t}p^u Tr_{\Delta,p^{2t}d'}(J_n)+\sum_{t=1}^u p^{u-t}  Tr_{\Delta,p^{2m-2u+4t}d'}(J_n)},\hbox{ if } m\geq u.
\end{cases}
\end{equation*}

We recall from Theorem \ref{HB} and Theorem \ref{DC} that it  holds for $p\nmid \Delta$,
\begin{equation*}\label{pre}(\mathfrak D\circ \Psi)(f_d|pT_{1/2}(p^{2}))=((\mathfrak D\circ \Psi)(f_d))|T_2(p).
\end{equation*}
Let $m$ be a positive integer.  Using mathematical induction on $m$ along with the following recursive relations satisfied by the Hecke operators on the integral and half-integral weights
$$T_k(p^{m+1})=T_k(p)T_k(p^m)-p^{k-1}T_k(p^{m-1}),$$
$$T_\k(p^{2m+2})=T_\k(p^2)T_\k(p^{2m})-p^{2\k-2}T_\k(p^{2m-2}),$$
we obtain 
\begin{equation}\label{ap1}(\mathfrak D\circ \Psi)(f_d|p^mT_{1/2}(p^{2m}))=((\mathfrak D\circ \Psi)(f_d))|T_2(p^m).
\end{equation}
Applying \eqref{ap} and \eqref{ft} to \eqref{ap1}, we derive the following relation for the twisted traces of singular moduli pertaining to prime powers.
\begin{thm}\label{cong} Let $p$ be a prime not dividing a discriminant $\Delta>1$ and $m, n$ be  non-negative integers with $n>0$. For a negative discriminant $-d$, we write $d=p^{2u}d'$ where $p^2\nmid d'$ and $u\geq 0$. Then when $\ell={\rm min}({\rm ord}_p(n),m)$, we have
\begin{equation*}
\displaystyle\sum_{t=0}^{\ell} p^{t}Tr_{\Delta,d}\left(J_\frac{p^m n}{p^{2t}}\right) =\begin{cases}\displaystyle \sum_{t=0}^m p^{m-t}  Tr_{\Delta,p^{2u-2m+4t}d'}(J_n),\hbox{ if } 0\leq m<u,\\
\displaystyle\sum_{t=0}^{m-u}  \left(\frac{-d'}{p}\right)^{m-u-t}p^u Tr_{\Delta,p^{2t}d'}(J_n)+\displaystyle\sum_{t=1}^u p^{u-t}  Tr_{\Delta,p^{2m-2u+4t}d'}(J_n),\hbox{ if }  m \geq u.
\end{cases}
\end{equation*}
\end{thm}
Assume $A_m(n,d)$ denotes the coefficient of $q^n$ of the image of $f_d$ under the Hecke operator of index $m^2$ and weight $1/2$. Then $A_m(n,d)$ are integers for all $m, n$ and $d$ and by \cite[Eq.(25)]{Zagier}, 
\begin{equation*}\label{tra}
Tr_{\Delta,d}(J_{n})=\sqrt{\Delta}A_n(\Delta,d).
\end{equation*}
Hence if $p\nmid d$ and $m=1$, it holds by Theorem \ref{cong},
$$A_{pn}(\Delta,d)-\left(\frac{-d}{p}\right) A_n(\Delta,d)-A_n(\Delta,dp^2)\equiv 0\pmod p.$$
More generally, we obtain the following corollary of Theorem \ref{cong}.
\begin{cor}\label{Aeq} Let $p$ be a prime not dividing a discriminant $\Delta>1$ and $m,n$ be non-negative integers with $n>0$. For a negative discriminant $-d$, we assume ${\rm ord}_p(d)<2$. Then for $\ell={\rm min}({\rm ord}_p(n),m)$, we have
\begin{equation*}
\displaystyle \sum_{t=0}^{\ell}p^{t}A_\frac{p^m n}{p^{2t}}(\Delta,d) = \sum_{t=0}^m  \left(\frac{-d}{p}\right)^{m-t} A_n(\Delta,p^{2t}d).
\end{equation*}
\end{cor}
The congruences for $Tr_{\Delta,d}(J_n)$ or $A_m(\Delta,d)$ followed by Theorem \ref{cong} are different with existing congruences such as those found in \cite{Guer2007} and \cite{Ahlgren}. (See   \cite{Ahlgren} for details and more references.)

%Guerzhoy showed that if $\Delta$ and $-d$ are fundamental discriminants with $(\frac{-d}{p})=(\frac{\Delta}{p})$, then for any positive integers $m$ and $n$,
%\begin{equation}\label{gu}
%A_n(\Delta,p^{2m}d)=p^m A_n(p^{2m}\Delta,d).
%\end{equation}
%Hence if the discriminants $\Delta $ and $d$ satisfy the conditions in Corollary \ref{Aeq} with $p||d$ and $(\frac{-d}{p})=(\frac{\Delta}{p})$, then for any non-negative integer $m$, we have
%\begin{equation}\label{Acong1}
%A_{p^m}(\Delta,d)\equiv 0\pmod {p^m}.
%\end{equation}

%
\subsection{Genus one case}

Throughout this subsection we assume that
$N$ is prime and the genus of $\Gamma_0(N)$ is one. There are three such values, $N\in \{11, 17, 19 \}$, for which the genus of $\Gamma_0^+(N)$, the group generated by $\G_0(N)$ and the Fricke involution $W_N:=\sm 0&-1\\ N&0\esm$,  is zero. 

Congruences for singular moduli of $f\in M_0^!(\G_0(N))$ when genus of $\G_0(N)$ is zero have been discussed in numerous literature, including \cite{AO, Jenkins, Osburn}. 
The motivation for this subsection is to treat higher genus cases, but for the sake of expediency, we will only address the instance of genus 1 case. Moreover, if the genus of  $\Gamma_0(N)$ is one, then we can use the useful property that the genus of $\Gamma_0^+(N)$ is zero.

%Let $M_k^!(N)$ be the space of weakly holomorphic modular forms of weight $k$ for $\Gamma_0(N)$. 
For $\epsilon \in \{-1, 1\}$, we let $M_k^{!,{\rm sgn}(\epsilon)}(N)$ be the subspace of $M_k^!(N)$ consisting of all eigenforms of $W_N$ with eigenvalue $\epsilon$. 
The canonical basis in \cite{CK, CKL} of $M_0^{!,{\rm sgn}(\epsilon)}(N)$ consists of the form $f_{N,m}^{{\rm sgn}(\epsilon)}$ whose Fourier expansion is given by 
$$
f_{N,m}^{{\rm sgn}(\epsilon)}=q^{-m}+\sum_{n> m_0^{{\rm sgn}(\epsilon)}} a_N^{{\rm sgn}(\epsilon)} (m,n) q^n 
$$
for every integer $m\ge -m_0^{{\rm sgn}(\epsilon)}$, where $m_0^+=0$ and $m_0^-=-2$. 
In \cite{CK, CKL} such a basis is explicitly constructed and it is shown that
$a_N^{{\rm sgn}(\epsilon)} (m,n)\in \mathbb{Z}$. 
 We then observe that for $m\ge 2$ the canonical basis elements $\mathfrak{f}_{N,m}$ of $M_0^\sharp(N)$ defined in Section \ref{sec2} are given by
\begin{equation} \label{R1}
\mathfrak{f}_{N,m}=\frac{ f_{N,m}^+ +f_{N,m}^- +a_N^-(m,-1)f_{N,1}^+ - a_N^-(m,0)}{2}.
\end{equation}

For $f\in M_0^{!,+}(N)$, if we define 
$$Tr_{\Delta,d,N}^+(f):=\sum_{Q\in\mathcal Q_{d\Delta,N}/\bar{\Gamma}_0^+(N)}\frac{\chi_\Delta(Q)}{|\bar{\Gamma}_0^+(N)_Q|}f(\a_Q),$$
then we find that 
\begin{equation} \label{R2}
Tr_{\Delta,d,N}(f)=
\begin{cases}
2Tr_{\Delta,d,N}^+(f), & \hbox{ \, if }  f\in M_0^{!,+}(N) \\
0, & \hbox{ \, if }  f\in M_0^{!,-}(N).
\end{cases}
\end{equation}
It then follows from \eqref{R1} and \eqref{R2} that for all $m\ge 2$, 
\begin{equation*} \label{trtrp}
Tr_{\Delta,d,N}(\mathfrak{f}_{N,m})=Tr_{\Delta,d,N}^+(f_{N,m}^+)+a_N^-(m,-1) Tr_{\Delta,d,N}^+(f_{N,1}^+)-a_N^-(m,0) Tr_{\Delta,d,N}^+(1).
\end{equation*}
Moreover, we have 
$\frac{1}{\sqrt{\Delta}}Tr_{\Delta,d,N}(\mathfrak{f}_{N,m}) \in \mathbb{Z} $ 
since the coefficients $a_N^-(m,-1), a_N^-(m,0)\in \mathbb{Z}$ and one has
$\frac{1}{\sqrt{\Delta}}Tr_{\Delta,d,N}^+(f_{N,n}^+) \in \mathbb{Z} $ which
follows from
\cite[Proposition 2.1]{Kim2009}.

Now we consider the modular form $f_{d,N}\in M^!_{1/2}(4N)$ that is uniquely determined by its Fourier expansion of the form 
$$
f_{d,N}=q^{-d}+\sum_{\substack{n\geq 1\\n\equiv\square\ (4N)}}A(n,d;N)q^n.
$$  
By \cite{CC}, we have an isomorphism from $M^!_{1/2}(N)$ to  $M^!_{1/2, \rho_N}$. So we may treat $f_{d,N}$ as an element of  $M^!_{1/2, \rho_N}$. Then by \cite{Kim2006, Kim2009}, the generalized Borcherds product of $f_{d,N}$ turns out to be
\begin{equation} \label{BP}
\Psi(f_{d,N})=\Psi_{\Delta,r}(\t,f_{d,N})=\prod_{Q\in \mathcal Q_{\Delta d, N}/\bar{\Gamma}_0^+(N)}(f_{N,1}^+(\t)-f_{N,1}^+(\alpha_Q))^{\chi_\Delta(Q)/|\bar{\Gamma}_0^+(N)_Q|},
\end{equation}
which is a weight $0$ meromorphic modular form for $\G_0(N)$.
Hence
\begin{equation}\label{ttr1}
\mathfrak D(\Psi(f_{d,N}))=\frac{\Theta(\Psi(f_{d,N}))}{\Psi(f_{d,N})}=\frac{1}{2\pi i}\sum_{Q\in\mathcal Q_{d\Delta,N}/\bar{\G}_0^+(N)}\frac{\chi_\Delta(Q)}{|\bar{\G}_0^+(N)_Q|}\frac{(f_{N,1}^+)'(\t)}{f_{N,1}^+(\t)-f_{N,1}^+(\a_Q)}.
\end{equation}
In conjunction with the following result \cite[Theorem 1.1]{Ye} 
$$
\frac{1}{2\pi i}\frac{(f_{N,1}^+)'(\t)}{f_{N,1}^+(\t)-f_{N,1}^+(z)}=-\sum_{n=0}^\i f_{N,n}^+(z) q^n,
$$
\eqref{ttr1} yields 
\begin{equation}\label{ttr2}
(\mathfrak D\circ \Psi)(f_{d,N})=-\sum_{n=0}^\i Tr_{\Delta,d,N}^+(f_{N,n}^+)q^n.
\end{equation}

Furthermore, we observe from \eqref{BP} that the divisor of $\Psi(f_{d,N})$ is given by 
\begin{equation*}\label{div1}
%{\rm div}(\Psi(f_{d,N}))=
\sum_{Q\in \mathcal Q_{\Delta d, N}/\bar{\G}_0^+(N)} \frac{\chi_\Delta(Q) [(\alpha_Q) + (W_N \alpha_Q)]}{|\bar{\G}_0^+(N)_Q|}
= \sum_{Q\in \mathcal Q_{\Delta d, N}/\bar{\G}_0(N)}\frac{ \chi_\Delta(Q) (\alpha_Q)}{|\bar{\G}_0(N)_Q|},
\end{equation*}
which implies that
\begin{equation} \label{div2}
\sum_{z\in \F_N} \frac{{\rm ord}_z(f)}{e_{N,z}}  F_{N,z}(\t)=\sum_{n=0}^\i Tr_{\Delta,d,N}(\mathfrak{f}_{N,n})q^n.
\end{equation}
Using \eqref{ttr1}, \eqref{ttr2} and \eqref{div2} in \eqref{divisor}, we find that 

\begin{equation} \label{div3}
\sum_{n=0}^\i Tr_{\Delta,d,N}(\mathfrak{f}_{N,n})q^n-\sum_{n=0}^\i Tr_{\Delta,d,N}^+(f_{N,n}^+)q^n  \in M_2(N).
\end{equation}
Since the space $M_2(N)$ is spanned by $\{g_{N,-1}, g_{N,0}\}$, where 
$\displaystyle{g_{N,-1}:=\sum_{n=1}^\i\a_nq^n}$ is the unique normalized cusp form and 
$\displaystyle{g_{N,0}(\t):=\frac{E_2(\t)-N E_2(N\t)}{1-N}+ \left(\frac{24}{1-N}\right) g_{N,-1}=1+O(q^2),}$
the left hand side of \eqref{div3} is written as
$$
 \sum_{n=0}^\i Tr_{\Delta,d,N}(\mathfrak{f}_{N,n})q^n-\sum_{n=0}^\i Tr_{\Delta,d,N}^+(f_{N,n}^+)q^n=Tr_{\Delta,d,N}^+(1) g_{N,0}-Tr_{\Delta,d,N}^+(f_{N,1}^+)g_{N,-1}.
$$ 
Here we used \eqref{R2} and the fact $\mathfrak{f}_{N,1}=0$.
Consequently, \eqref{ttr2} shows that
\begin{equation} \label{star}
(\mathfrak D\circ \Psi)(f_{d,N})=-\sum_{n=0}^\i Tr_{\Delta,d,N}(\mathfrak{f}_{N,n})q^n
+Tr_{\Delta,d,N}^+(1) g_{N,0}-Tr_{\Delta,d,N}^+(f_{N,1}^+)g_{N,-1}.
\end{equation}

%{\color{green}Recall the Hecke invariance in Theorem \ref{HB} and Theorem \ref{DC}: for $p\nmid \Delta$,
%\begin{equation}\label{pre}
%((\mathfrak D\circ \Psi)(f_{d,N}))|T_2(p)=(\mathfrak D\circ \Psi)(f_{d,N}|pT_{1/2}(p^{2})).
%\end{equation}}
%
Following the same procedure as in Section 4.1, we deduce from our Hecke equivariance that for $n\ge 0$ and $p\nmid \Delta$, 
\begin{equation} \label{HEP}
Tr_{\Delta,d,N}^+(f_{N,pn}^+)+pTr_{\Delta,d,N}^+(f_{N,n/p}^+)=pTr_{\Delta,d/p^2,N}^+(f_{N,n}^+)+\left(\frac{-d}{p}\right)Tr_{\Delta,d,N}^+(f_{N,n}^+)+Tr_{\Delta,dp^2,N}^+(f_{N,n}^+).
\end{equation}
We also have that 
\begin{align}\label{Hecke}
(\mathfrak D\circ \Psi)(f_{d,N})|T_2(p)
&=(\mathfrak D\circ \Psi)(f_{d,N}|pT_{1/2}(p^{2})) \\
&= (\mathfrak D\circ \Psi)\left(pf_{d/p^2,N}+\left( \frac{-d}{p}\right)f_{d,N}+f_{dp^2,N}\right) \notag \\
&=p(\mathfrak D\circ \Psi)(f_{d/p^2,N})+\left( \frac{-d}{p}\right)(\mathfrak D\circ \Psi)(f_{d,N})+(\mathfrak D\circ \Psi)(f_{dp^2,N}).\notag 
\end{align}
We observe that
$g_{N,0}|T_2(p)=(1+p)g_{N,0}+\frac{24}{1-N} (\alpha_p-1-p)g_{N,-1}$ and
$g_{N,-1}|T_2(p)=\alpha_p g_{N,-1}$.
We then have from \eqref{star} and \eqref{Hecke} that
\begin{align}
& -\sum_{n=0}^\i (Tr_{\Delta,d,N}(\mathfrak{f}_{N,pn})+pTr_{\Delta,d,N}(\mathfrak{f}_{N,n/p}))q^n \label{2star} \\
& +Tr_{\Delta,d,N}^+(1) \left((1+p)g_{N,0}+\frac{24}{1-N} (\alpha_p-1-p)g_{N,-1}\right)- Tr_{\Delta,d,N}^+(f_{N,1}^+)\alpha_p g_{N,-1}  \notag \\
=& -\sum_{n=0}^\i  \left(pTr_{\Delta,d/p^2,N}(\mathfrak{f}_{N,n}) +\left( \frac{-d}{p}\right)Tr_{\Delta,d,N}(\mathfrak{f}_{N,n})+Tr_{\Delta,dp^2,N}(\mathfrak{f}_{N,n}) \right) q^n \notag \\
&+ \left(pTr_{\Delta,d/p^2,N}^+(1) +\left( \frac{-d}{p}\right)Tr_{\Delta,d,N}^+(1)+Tr_{\Delta,dp^2,N}^+(1) \right)g_{N,0} \notag \\
&- \left(pTr_{\Delta,d/p^2,N}^+(f_{N,1}^+) +\left( \frac{-d}{p}\right)Tr_{\Delta,d,N}^+(f_{N,1}^+)+Tr_{\Delta,dp^2,N}^+(f_{N,1}^+) \right) g_{N,-1} \notag. 
\end{align}
Comparing the constant terms on both sides of \eqref{2star} and applying \eqref{R2} yields
\begin{equation} \label{cons}
(1+p)Tr_{\Delta,d,N}^+(1)=pTr_{\Delta,d/p^2,N}^+(1) +\left( \frac{-d}{p}\right)Tr_{\Delta,d,N}^+(1)+Tr_{\Delta,dp^2,N}^+(1).
\end{equation}
Applying \eqref{HEP} with $n=1$ and \eqref{cons} to \eqref{2star}, we obtain the equality
% Meanwhile, if we compare  $n=1$ in \eqref{HEP}, we get an identity as follows:
%$$
%Tr_{\Delta,d,N}^+(f_{N,p}^+)=pTr_{\Delta,d/p^2,N}^+(f_{N,1}^+) +\left(\frac{-d}{p}%%\right)Tr_{\Delta,d,N}^+(f_{N,1}^+)+Tr_{\Delta,dp^2,N}^+(f_{N,1}^+) .
%$$
%Then \eqref{2star} is simplified as follows:
\begin{align}
& -\sum_{n=0}^\i (Tr_{\Delta,d,N}(\mathfrak{f}_{N,pn})+pTr_{\Delta,d,N}(\mathfrak{f}_{N,n/p}))q^n \label{3star} \\
=& -\sum_{n=0}^\i  \left(pTr_{\Delta,d/p^2,N}(\mathfrak{f}_{N,n}) +\left( \frac{-d}{p}\right)Tr_{\Delta,d,N}(\mathfrak{f}_{N,n})+Tr_{\Delta,dp^2,N}(\mathfrak{f}_{N,n}) \right) q^n \notag \\ 
&+ \left(  -Tr_{\Delta,d,N}^+(1)\frac{24}{1-N} (\alpha_p-1-p)+ Tr_{\Delta,d,N}^+(f_{N,1}^+)\alpha_p- Tr_{\Delta,d,N}^+(f_{N,p}^+)  \right) g_{N,-1} \notag. 
\end{align}
Now we compare the coefficients of  $q$ on both sides of \eqref{3star} to find that 
\begin{equation*} \label{ones}
Tr_{\Delta,d,N}(\mathfrak{f}_{N,p})=Tr_{\Delta,d,N}^+(1)\frac{24}{1-N} (\alpha_p-1-p)- Tr_{\Delta,d,N}^+(f_{N,1}^+)\alpha_p+Tr_{\Delta,d,N}^+(f_{N,p}^+),
\end{equation*}
by which \eqref{3star} is simplified to
\begin{align}
& \sum_{n=0}^\i (Tr_{\Delta,d,N}(\mathfrak{f}_{N,pn})+pTr_{\Delta,d,N}(\mathfrak{f}_{N,n/p}))q^n \label{4star} \\
=& \sum_{n=0}^\i  \left(pTr_{\Delta,d/p^2,N}(\mathfrak{f}_{N,n}) +\left( \frac{-d}{p}\right)Tr_{\Delta,d,N}(\mathfrak{f}_{N,n})+Tr_{\Delta,dp^2,N}(\mathfrak{f}_{N,n}) \right) q^n +Tr_{\Delta,d,N}(\mathfrak{f}_{N,p})  g_{N,-1} \notag. 
\end{align}
Finally, by comparing the Fourier coefficients $q^n$ for $n\geq 2$ on both sides of \eqref{4star}, we obtain the following identity: for $n\ge 2$,
\begin{align*}
Tr_{\Delta,d,N}(\mathfrak{f}_{N,pn})+pTr_{\Delta,d,N}(\mathfrak{f}_{N,n/p})
&=pTr_{\Delta,d/p^2,N}(\mathfrak{f}_{N,n}) +\left( \frac{-d}{p}\right)Tr_{\Delta,d,N}(\mathfrak{f}_{N,n})+Tr_{\Delta,dp^2,N}(\mathfrak{f}_{N,n})\\
& \, \, + \alpha_n Tr_{\Delta,d,N}(\mathfrak{f}_{N,p}).
\end{align*}
We summarize the results above as a theorem.
\begin{thm}\label{trcong3} Let $p$ be a prime not dividing a discriminant $\Delta>1$ nor $N$ and $d$ be a positive integer such that $-d$ is congruent to a square modulo $4N$. Then we have
\begin{enumerate}
\item $\displaystyle{(1+p)Tr_{\Delta,d,N}^+(1)=pTr_{\Delta,d/p^2,N}^+(1) +\left( \frac{-d}{p}\right)Tr_{\Delta,d,N}^+(1)+Tr_{\Delta,dp^2,N}^+(1),}$
\item $\displaystyle{Tr_{\Delta,d,N}(\mathfrak{f}_{N,p})=Tr_{\Delta,d,N}^+(1)\frac{24}{1-N} (\alpha_p-1-p)- Tr_{\Delta,d,N}^+(f_{N,1}^+)\alpha_p+Tr_{\Delta,d,N}^+(f_{N,p}^+),}$
\item for integers $n\geq 2$, \begin{align*}
Tr_{\Delta,d,N}(\mathfrak{f}_{N,pn})+pTr_{\Delta,d,N}(\mathfrak{f}_{N,n/p})
&=pTr_{\Delta,d/p^2,N}(\mathfrak{f}_{N,n}) +\left( \frac{-d}{p}\right)Tr_{\Delta,d,N}(\mathfrak{f}_{N,n})\\
&+Tr_{\Delta,dp^2,N}(\mathfrak{f}_{N,n}) + \alpha_n Tr_{\Delta,d,N}(\mathfrak{f}_{N,p}).
\end{align*}
\end{enumerate}
\end{thm}

There is an immediate corollary to Theorem \ref{trcong3}.

\begin{cor}\label{trcong4} Let $p$ be a prime not dividing a discriminant $\Delta>1$ nor $N$ and $d$ be a positive integer such that $-d$ is congruent to a square modulo $4N$. Then for a non-negative integer $n\neq 1$, we have
$$
\frac{Tr_{\Delta,d,N}(\mathfrak{f}_{N,pn})}{\sqrt{\Delta}}\equiv
\left( \frac{-d}{p}\right)\frac{Tr_{\Delta,d,N}(\mathfrak{f}_{N,n})}{\sqrt{\Delta}}+\frac{Tr_{\Delta,dp^2,N}(\mathfrak{f}_{N,n})}{\sqrt{\Delta}}+ \alpha_n \frac{ Tr_{\Delta,d,N}(\mathfrak{f}_{N,p})}{\sqrt{\Delta}}
\pmod{p}.
$$
\end{cor}

 In particular, we establish a congruence for the twisted class numbers from Theorem \ref{trcong3}(1). 

\begin{cor}Let $p$ be a prime not dividing a discriminant $\Delta>1$ nor $N$ and $d$ be a positive integer such that $-d$ is congruent to a square modulo $4N$. Then
$$
H_N(\Delta,dp^2)\equiv
\begin{cases}
0 \pmod p & \ \text{if}\  \lt(\frac{-d}{p}\rt)=1,\\
H_N(\Delta,d) \pmod p &\  \text{if} \ \lt(\frac{-d}{p}\rt)=0,\\
2H_N(\Delta,d) \pmod p & \ \text{if}\  \lt(\frac{-d}{p}\rt)=-1.
\end{cases}
$$

\end{cor}

\end{document}